\documentclass[letterpaper, 10 pt, conference]{ieeeconf}  

\IEEEoverridecommandlockouts                              

\usepackage{graphics} 
\usepackage{epsfig} 
\usepackage{mathptmx} 
\usepackage{times} 
\usepackage{amsmath} 
\usepackage{amssymb}  
\newcommand{\ltwo}{{\ell^2}}

\newtheorem{theorem}{Theorem}
\newtheorem{lemma}{Lemma}

\newtheorem{assumption}{Assumption}

\title{\LARGE \bf
A Dynamic Observer for a Class of Infinite-Dimensional Vibrating Flexible Structures}

\author{Alexander Zuyev$^{1,2,3}$ and Julia Kalosha$^{2,3}$
\thanks{$^{1}$Otto von Guericke University Magdeburg, Universitätsplatz 2, 39106 Magdeburg, Germany
        {\tt\small zuyev@mpi-magdeburg.mpg.de}}%
\thanks{$^{2}$Institute of Applied Mathematics \& Mechanics, National Academy of Sciences of Ukraine, G.\,Batiuka St., 19, Sloviansk, Ukraine 84100
        {\tt\small julykucher@gmail.com}}%
\thanks{$^{3}$Max Planck Institute for Dynamics of Complex Technical Systems, Sandtorstra{\ss}e 1, 39106 Magdeburg, Germany
 \newline The first author was supported by the German Research Foundation (DFG) under Grant ZU 359/2-1.
 The second author is partially supported by Grant EFDS-FL2-08 of the European Federation of Academies of Sciences and Humanities (ALLEA).
        }%
}

\begin{document}

\maketitle
\thispagestyle{empty}
\pagestyle{empty}

\begin{abstract}
Infinite-dimensional control systems with outputs are considered in the Hamiltonian formulation with generalized coordinates.
An explicit scheme for constructing a dynamic observer for this class of systems is proposed with arbitrary gain coefficients.
Sufficient conditions for the convergence of the constructed observer are obtained on the basis of the invariance principle.
This result is applied to a flexible beam model attached to a mass-spring system with lumped and distributed actuators.
The estimation error decay is illustrated with numerical simulations of finite-dimensional approximations of the observer dynamics.
\end{abstract}

\section{INTRODUCTION}

Since the full state of a dynamical system is not available for direct measurement in many applied problems,
it becomes crucial for optimal control and stabilization to estimate the state in a reliable manner employing observations of a limited number of system outputs.
A classical method of estimating the state vector using the inputs and outputs, proposed by D.~Luenberger~\cite{L1964}, has been widely extended for various classes of control systems, including distributed parameter systems~\cite{ZBPR2015}--\cite{YW2019}.
Note that the Kalman observability rank condition or Hautus test are not directly applicable for control systems with infinite degrees of freedom. So, powerful operator techniques have been developed to unfold infinite-dimensional modifications of classical methods (see, e.g.,\cite{DM1972}--
\cite{DSX2006}).

In recent years, a variety of results in observer design and observer-based control have been obtained for specific engineering problems.
A Luenberger-type observer is proposed in~\cite{ZS2007} with the use of the Lyapunov method for a flexible structure modeled by the Timoshenko beam.
A feedback control scheme based on the disturbance observer is designed for a flexible spacecraft subject to external disturbance, input magnitude, and rate constraints in~\cite{LLW2018}.
A reduced-order nonlinear state observer is proposed in~\cite{PRT2018} for a flexible-link multibody system by using model reduction techniques.
In the paper~\cite{MTG2020}, a nonlinear observer based control is presented for a rotating flexible structure modeled as the Timoshenko beam. The exponential convergence of the estimation scheme and the asymptotic stability of the closed-loop system are proved in this work.
A state estimator for the motion of a slewing flexible structure with rigid body is designed in~\cite{DBBSB2021}.
In~\cite{P2022}, the Luenberger observer is used for estimating the internal state of a ``black-box'' linear system, and the optimal feedback control is proposed. Stability under the Luenberger observer parametrization is proved based Lyapunov stability theory.
In all these works results are experimentally validated.

The observability and controllability of diagonal systems with the infinite-dimensional state and finite-dimensional output are investigated in~\cite{JZ2001f},~\cite{JZ2001_1}. In these papers, necessary and sufficient conditions for the exact observability are extended to the systems that are not necessarily exponentially stable. The above results are related to an infinite-dimensional version of the Hautus test proposed by D.L.~Russell and G.~Weiss in~\cite{RW1994}. Note that in the works of B.~Jacob and H.~Zwart, the exact observability problem is solved for diagonal systems without input and with the output signal of class $L^2(0,\infty)$.

In the present work, we will propose an explicit observer design for a controlled Hamiltonian system
without the requirement that the output signal is in $L^2(0,\infty)$. While our assumptions are weaker than in~\cite{JZ2001f},~\cite{JZ2001_1}, we do not formulate exact observability conditions and prove that the observation error tends to zero only asymptotically.

This paper may be contemplated as a further extension of the approach developed in~\cite{ZK2021} to distributed parameter systems.
In Section~\ref{sec_Hamiltonian}, we present a linear control system with the parameterized Hamiltonian operator and a finite-dimensional output.
It is shown that the considered system generates a $C_0$-semigroup of operators in the real Hilbert space $\ltwo$  under a suitable assumption on the parameters.
Then we describe the observer design procedure in Section~\ref{sec_observer}.
Our stability analysis of the error dynamics is based on LaSalle's invariance principle (or an infinite-dimensional version of the Barbashin--Krasovskii theorem) in Section~\ref{sec_stability}.
A key part in the stability proof (proof of Theorem~1) requires the precompactness of trajectories, which is formally addressed in Section~\ref{sec_compact}.
Conditions for the observer convergence are discussed in terms of parameters of the control system.

Section~\ref{sec_beam} is devoted to a particular case, if the considered class of systems describes vibrations of the flexible beam with an attached mass discussed in~\cite{KZ2021}. In this case, parameters of the infinitesimal generator arise as eigenvalues of a fourth-order differential operator. It has been proved in~\cite{ZK2021} that the above operator is self-adjoint and positive, so its eigenvalues form a sequence of positive real numbers. Besides, it has been proved in~\cite{KZB2021} that all roots of the corresponding characteristic equation are simple and satisfy certain growth condition.
As a consequence of the eigenvalue distribution added to a natural assumption that at least one output coefficient is nonzero for each mode,
the observer convergence is established in Theorem~2.
The observation error dynamics for the considered beam model is illustrated by numerical simulations in Section~\ref{sec_simulations}.

%
%


\section{INFINITE-DIMENSIONAL CONTROL SYSTEM WITH OUTPUT}\label{sec_Hamiltonian}

Consider an infinite-dimensional control system of the form
\begin{equation}\label{linsys}
	\dot z=Az+Bu,\quad z\in H=\ell^2\times \ell^2,\quad u\in{\mathbb R}^{k+1},
\end{equation}
\begin{equation}\label{output}
	y =Cz,\quad y\in{\mathbb R}^r,\quad k,r\in \mathbb N,
\end{equation}
where the state vector $z=\left(
\begin{array}{c}
	\xi \\
	\eta \\
\end{array}
\right)$ consists of two components ${\xi=(\xi_1,\xi_2,\dots)^T\in \ell^2}$ and ${\eta=(\eta_1,\eta_2,\dots)^T\in \ell^2}$,
the control $u=(u_0,u_1,\dots,u_k)^T$ and the output $y=(y_1,\dots,y_r)^T$ are finite-dimensional vectors.
The inner product in the real Hilbert space $H$ is inherited from $\ell^2$:
$$
\left\langle \left(\begin{array}{c}
	\xi \\
	\eta \\
\end{array}\right),\left(\begin{array}{c}
	\tilde \xi \\
	\tilde \eta \\
\end{array}\right)\right\rangle_H =  \sum_{j=1}^\infty\left(\xi_j\tilde \xi_j + \eta_j\tilde \eta_j \right),
$$
and the operators $A:D(A)\to H$, $B:\mathbb R^{k+1}\to H$, $C:H\to \mathbb R^r$ are defined in terms of the $\xi$- and $\eta$-component of $z\in H$ as
$$
A: z=\begin{pmatrix}\xi \\ \eta \end{pmatrix} \mapsto  Az = \begin{pmatrix} \Omega \eta \\ -\Omega \xi \end{pmatrix},\; Bu = \begin{pmatrix}0\\ B_1 u \end{pmatrix},\; C z = C_1 \xi,
$$
where $\Omega = {\rm diag}(\omega_1,\omega_2,...),$ and the operators $B_1:\mathbb R^{k+1}\to \ltwo$, $C_1:\ltwo\to \mathbb R^r$ are given by their matrices
$$
B_1=\left(\begin{array}{cccc}
	b_{10} & b_{11} & \ldots & b_{1k} \\
	b_{20} & b_{21} & \ldots & b_{2k}\\
	\vdots & \vdots & \vdots & \vdots
\end{array}\right),\quad
\sum\limits_{j=1}^\infty b_{ji}^2<\infty,\; i=0,\dots,k,$$
$$
C_1=\left(\begin{array}{llc}
	c_{11} & c_{12} &  \ldots \\
	\vdots & \vdots  & \vdots \\
	c_{r1} & c_{r2} &  \ldots \end{array}\right),
\quad
\sum\limits_{j=1}^\infty c_{sj}^2<\infty,\; s=1,\dots,r.
$$
Thus,
$$D(A)=\left\{z=\left(
	\begin{array}{c}
		\xi \\
		\eta \\
	\end{array}
	\right)\in H:\sum\limits_{j=1}^\infty\omega_j^2(\xi_j^2+\eta_j^2)<\infty\right\}.
$$
System~\eqref{linsys} is considered as a mathematical model of vibrating flexible structures with force actuation and displacement measurement.
We assume that the spectrum of $A$ is purely imaginary without resonances, which is formalized in the following way.

\begin{assumption}\label{assumption1}
	The diagonal entries $\omega_j$ of $\Omega$ are mutually distinct positive real numbers such that
$${0<\omega_1<\omega_2<\dots <\omega_n <\dots}\;.$$
\end{assumption}

It will be shown in section~\ref{sec_beam} that Assumption~\ref{assumption1} is fulfilled for the particular system due to properties of the infinitesimal generator.

Unlike finite-dimensional ODE systems with smooth right-hand side, for which the conditions of existence and uniqueness of the Cauchy problem solution are clear, the question about well-posedness of the Cauchy problem for infinite-dimensional systems becomes one of the major tasks. For dealing with this the semigroup representation is commonly used.

\begin{lemma}\label{le_generator}
	The operator $A$
	generates a $C_0$-semigroup $\{e^{tA}\}_{t\geq0}$ on $H$.
\end{lemma}

\begin{proof}
	The domain $D(A)$ is dense in $H$.
	For any $z\in D(A)$, we have
	\begin{equation}\label{dissip}
		\langle Az, z\rangle_{H}\equiv 0,
	\end{equation}
    so the map $A$ is dissipative.
	
	The inverse operator $A^{-1}:H\to H$ is defined as follows:
	$$A^{-1}=\left(
	\begin{array}{cc}
		0 & -\tilde{A}_1 \\
		\tilde{A}_1 & 0 \\
	\end{array}
	\right),$$
	where $\tilde{A}_1={\rm diag}\left(\frac1{\omega_1},\frac1{\omega_2},\dots\right)$.
	
	It is easy to verify that the following estimate is fulfilled:
	$$\|A^{-1}z\|^2\leq\frac1{\omega_1^2}\;\|z\|^2,\quad\forall\; z\in H,$$
	so $A^{-1}$ is bounded, therefore the operator $A$ is closed in $H$.
	
	The resolvent $Res(A)$ of $A$ may be constructed as the map $Res(A)=(I-\lambda A)^{-1}:H\to H$ with the components
	\begin{equation*}
	Res(A)=\left(
	\begin{array}{cc}
		R_1 & R_2 \\
		-R_2 & R_1 \\
	\end{array}
	\right),
	\end{equation*}
	where
	$$\begin{aligned}
		R_1&={\rm diag}\left(\frac1{\lambda^2\omega_j^2+1},\;j=1,2,\dots\right),\\
		R_2&={\rm diag}\left(\frac{\lambda\omega_j}{\lambda^2\omega_j^2+1},\;j=1,2,\dots\right)
	\end{aligned}$$
	for some $\lambda>0$. Here and in the sequel, $I$ is the identity operator on $H$.
	
	As the transformation $z\mapsto Res(A)(z)$ is defined for any $z\in H$,
	the range of $I-\lambda A$ coincides with $H$. Thus, the map $A$ is maximal. Moreover, because of condition~\eqref{dissip}, $A$ is $m$-dissipative. Being densely defined, $m$-dissipative, and closed, the operator $A$ satisfies the conditions of the Lumer--Phillips theorem and thus generates a $C_0$-semigroup on $H$.
\end{proof}

\section{LUENBERGER OBSERVER DESIGN}\label{sec_observer}

Our aim is to construct a Luenberger-type observer for system~\eqref{linsys} in the form
\begin{equation}\label{observer}
	\dot{\bar z}(t)=(A-FC)\bar z(t)+Bu(t)+Fy(t),
\end{equation}
such that, for any initial conditions $z(0),\bar z(0)\in H$ and any admissible control $u: [0,+\infty)\to {\mathbb R}^{k+1}$, the corresponding solutions $z(t)$ and $\bar z(t)$ of~\eqref{linsys},~\eqref{output} and~\eqref{observer} satisfy the property
\begin{equation}\label{conv}
\|z(t)-\bar z(t)\|\to 0 \quad\; \text{as}\; t\to+\infty.
\end{equation}
For this purpose, we will define the operator $F:{\mathbb R}^{r}\to H$ as
$F=\left(
\begin{array}{c}
	f \\
	g \\
\end{array}
\right)$ with $f,g:{\mathbb R}^{r}\to \ltwo$ given by their matrices
$$f=\left(
\begin{array}{ccc}
	f_{11} & \dots & f_{1r} \\
	f_{21} & \dots & f_{2r} \\
	\vdots & \vdots & \vdots \\
\end{array}
\right)\quad \text{and}\quad
g=\left(
\begin{array}{ccc}
	g_{11} & \dots & g_{1r} \\
	g_{21} & \dots & g_{2r} \\
	\vdots & \vdots & \vdots \\
\end{array}
\right),$$
where
\begin{equation}\label{obsmatr}
	f_{js}=\gamma_sc_{sj},\quad g_{js}=0,\quad s=\overline{1,r},\quad j=1,2,\dots,
\end{equation}
and $\gamma_s>0$ are gain parameters.

Let us consider the observation error ${e(t)=z(t)-\bar z(t)}$ and denote
$e(t)=\left(
\begin{array}{c}
	\Delta \\
	\delta \\
\end{array}
\right)$, where ${\Delta=(\Delta_1,\Delta_2,\dots)^T}$, ${\delta=(\delta_1,\delta_2,\dots)^T}$. Then the error dynamics can be written as follows:
\begin{equation}\label{error}
	\dot e(t)=(A-FC)e(t).
\end{equation}
The property~\eqref{conv} is equivalent to
$$\|e(t)\|\to0\quad\text{as}\quad t\to\infty,$$
where $e(t)$ is the solution of~\eqref{error}.

\section{PRECOMPACTNESS OF THE TRAJECTORIES}\label{sec_compact}
In this section, we will construct the resolvent $Res(\hat A)$ of the operator $\hat A=A-FC$ and prove its compactness for some $\lambda>0$.
This result will be then applied in Section~\ref{sec_stability} for the proof of our main result (Theorem~1).

In order to construct $Res(\hat A)$, we solve the equation
\begin{equation}\label{erres}
	(\hat A-\lambda I)\left(
	\begin{array}{c}
		\Delta \\
		\delta \\
	\end{array}
	\right)=\left(
	\begin{array}{c}
		\bar\Delta \\
		\bar\delta \\
	\end{array}
	\right)
\end{equation}
with respect to $(\Delta,\delta)^T\in H$ for any given $(\bar\Delta,\bar\delta)^T\in H$, $\lambda>0$.

Let us define the matrix $$M=(M_{sp})=\lambda\gamma_p\sum\limits_{i=1}^\infty\frac{c_{si}c_{pi}}{\lambda^2+\omega_i^2}+{\bf\delta_{sp}},\quad s,p=\overline{1,r}$$
and its inverse $M^{-1}=(M^{-1}_{sp})_{s,p=\overline{1,r}}$.
The symbol ${\bf\delta_{sp}}$ denotes the Kronecker delta.

Note that, for small $\lambda>0$, the matrix $M$ is a small disturbance of the identity matrix, so the inverse matrix $M^{-1}$ exists and satisfies the norm estimate
\begin{equation*}
	\|M^{-1}\|\leq1+\mathcal{O}(\lambda) \quad \text{as}\;\lambda\to0.
\end{equation*}

The solution of~\eqref{erres} can be presented as follows:
\begin{equation}\label{erres_sol}
	\begin{aligned}
		\Delta_j&=-\frac1{\lambda^2+\omega_j^2}\left(\lambda\bar\Delta_j+\omega_j\bar\delta_j+\lambda\sum\limits_{s=1}^r\gamma_sc_{sj}\phi_s\right),\\
		\delta_j&=\frac1{\lambda^2+\omega_j^2}\left(\omega_j\bar\Delta_j-\lambda\bar\delta_j+\omega_j\sum\limits_{s=1}^r\gamma_sc_{sj}\phi_s\right).
	\end{aligned}
\end{equation}
The parameters $\phi_s$, $s=\overline{1,r}$, can be obtained from the following equation:
$\left(\phi_1,\dots,\phi_r\right)^T=M^{-1}\bar M$,
where $\bar M=-{\rm colon}\left(\sum\limits_{i=1}^\infty\frac{c_{si}}{\lambda^2+\omega_i^2}\left(\lambda\bar\Delta_i+\omega_i\bar\delta_i\right),\;s=\overline{1,r}\right)$.
So,
$$\phi_s=-\sum\limits_{p=1}^rM^{-1}_{sp}\sum\limits_{i=1}^\infty\frac{c_{pi}}{\lambda^2+\omega_i^2}\left(\lambda\bar\Delta_i+\omega_i\bar\delta_i\right)\quad s=\overline{1,r}.$$

The resolvent of $\hat A$ is acting  as
\begin{equation*}
	Res(\hat A)=\left(
	\begin{array}{cc}
		R^1 & R^2 \\
		R^3 & R^4 \\
	\end{array}
	\right),
\end{equation*}
where
\begin{equation*}
	\begin{aligned}
		&R^1_{ji}=\frac\lambda{\lambda^2+\omega_j^2}\left(\frac\lambda{\lambda^2+\omega_i^2}\sum\limits_{s,p=1}^r\gamma_sM^{-1}_{sp}c_{sj}c_{pi}-{\bf\delta_{ji}}\right),\\
		&R^2_{ji}=\frac1{\lambda^2+\omega_j^2}\left(\frac{\lambda\omega_i}{\lambda^2+\omega_j^2}\sum\limits_{s,p=1}^r\gamma_sM^{-1}_{sp}c_{sj}c_{pi}-\omega_j{\bf\delta_{ji}}\right),\\
		&R^3_{ji}=-\frac{\omega_j}\lambda R^1_{ji},\\
		&R^4_{ji}=-\frac1{\lambda^2+\omega_j^2}\left(\frac{\omega_j\omega_i}{\lambda^2+\omega_i^2}\sum\limits_{s,p=1}^r\gamma_sM^{-1}_{sp}c_{sj}c_{pi}+\lambda{\bf\delta_{ji}}\right).
	\end{aligned}
\end{equation*}

As Assumption~\ref{assumption1} is not be sufficient to guarantee the convergence of series appearing in~\eqref{erres_sol}, we introduce the following
\begin{assumption}\label{assumption_cnv}
	The series $\sum\limits_{i=1}^\infty\frac1{\omega_i^2}$ is convergent.
\end{assumption}

\begin{lemma}
	If Assumption~\ref{assumption_cnv} is fulfilled, then each positive semitrajectory $\{e(t)\}_{t\ge 0}$
	of system~\eqref{error} is precompact in $H$.
\end{lemma}

\begin{proof}
	If $\lambda>0$ is small enough and Assumption~\ref{assumption_cnv} is fulfilled, then formulas~\eqref{erres_sol} define  the resolvent of $\hat A$ as an operator that maps the vector $(\bar\Delta,\bar\delta)^T\in H$ to ${(\Delta,\delta)^T\in H}$.
	
	Let us prove its compactness.
	For this purpose we consider the Hilbert--Schmidt norm of the operator ${(a_{ij})=Res(\hat A)}$,
	$$\|a_{ij}\|=\sqrt{\sum\limits_{i=1}^\infty\sum\limits_{j=1}^\infty(a_{ij})^2}.$$
	After performing several estimates based on the Cauchy--Schwarz inequality, we conclude that $\|Res(\hat A)\|^2$
	does not exceed the following value:
	\begin{equation}\label{erres_est}
		2\sum\limits_{j,i=1}^\infty\frac1{\omega_j^2}\left(\frac1{\omega_i^2}\left(\sum\limits_{s,p=1}^r\gamma_sM^{-1}_{sp}c_{sj}c_{pi}\right)^2+2\right).
	\end{equation}
As $|M^{-1}_{sp}|\leq\|M^{-1}\|$, the estimate~\eqref{erres_est} of $\|Res\hat A\|^2$ is finite, while $c_{sj}$ and $\omega_j$ satisfy Assumption~\ref{assumption_cnv} for every $s=\overline{1,r}$, $j=1,2,\dots$ .
As the Hilbert--Schmidt norm of the resolvent is finite, $Res(\hat A):H\to H$ is a compact operator.
Because of the Dafermos--Slemrod (see~\cite{DS1973}), the positive trajectories of system~\eqref{error} are precompact in $H$
as the resolvent ${Res(\hat A)=(\hat A-\lambda I)^{-1}}$ is compact for small enough $\lambda>0$.
\end{proof}

\section{ASYMPTOTIC STABILITY}\label{sec_stability}
\begin{assumption}\label{assumption3}
	The only invariant subspace of ${\rm Ker}\,C$ under the action of the semigroup $\{e^{tA}\}_{t\geq0}$ is the singleton $\{0\}$.
\end{assumption}

\begin{assumption}\label{assumption4}
	For each $j\in{\mathbb N}$,
	there exists an $s=\overline{1,r}$ such that
	\begin{equation}\label{outputmatr}
		c_{sj}\neq0.
	\end{equation}
\end{assumption}

\begin{theorem}\label{th_observer_linsys}
	Let Assumptions~\ref{assumption1}--\ref{assumption4} be satisfied, and
	let the components of the operator $F$ be defined by~\eqref{obsmatr}.
	Then the trivial solution $e=0$ of~\eqref{error} is asympto\-tically stable.
\end{theorem}

\begin{proof}
	Consider the following positive definite quadratic form on $H$:
	$${\cal W}(e)=\sum\limits_{j=1}^\infty(\Delta_j^2+\delta_j^2),$$
	and calculate its time derivative along the trajectories of system~\eqref{error}:
	$$\dot{\cal W}(e)=-2\sum\limits_{j=1}^\infty\sum\limits_{i=1}^\infty\sum\limits_{s=1}^r c_{si}\,\Delta_i\left(f_{js}\Delta_j+g_{js}\delta_j\right).$$
	If $f_{js}$ and $g_{js}$ are defined by~\eqref{obsmatr}, then
	$$\dot{\cal W}(e)=-2\sum\limits_{s=1}^r\gamma_s\left(\sum\limits_{j=1}^\infty c_{sj}\Delta_j\right)^2\leq0\;\;\text{for all}\;
	\left(
	\begin{array}{c}
		\Delta \\
		\delta \\
	\end{array}
	\right)\in D(A).$$
	According to Lyapunov's theorem, the solution $e(t)\equiv0$ of~\eqref{error} is stable.
	
	In order to prove its asymptotic stability, consider the set $S=\{e(t):\dot{\cal W}(e)\equiv0\}$, where $e(t)$ is the solution of~\eqref{error}. From the equality
	\begin{equation}\label{ker}
		Ce(t)=0
	\end{equation}
	and the structure of $S$, we deduce that $S={\rm Ker}\,C$.
	
	Assumption~\ref{assumption3} now appears as the condition of   LaSalle's theorem~\cite{L1976} (see also~\cite{Z2003,Z2006} for the semigroup formulation), according to which the trivial solution of the error dynamics~\eqref{error} is asymptotically stable.
\end{proof}
The applicability of Assumption~3 will be checked under a specific choice of the parameters appearing in $A$ and $C$ for a particular class of flexible structures.

\section{FLEXIBLE BEAM VIBRATIONS}\label{sec_beam}
System~\eqref{linsys} represents the operator form of a wide class of mathematical models of controlled flexible structures.
In the sequel, we will consider a particular case of system~\eqref{linsys} resulting from the modal analysis of a vibrating flexible beam with an attached point mass.
 A complete description of this plant can be found in~\cite{KZ2021}.
The equation of motion of this system has been derived by using Hamilton's principle, which yields the following relation:
\begin{equation}\label{VF}
	\begin{aligned}
		\int_0^l \left(\rho \ddot w \delta w+ \left(EI w'' -\sum_{i=1}^ku_i\psi_i\right)\delta w''\right)dx&\\
		+\left.\left(m\ddot w+\varkappa w - u_0\right)\delta w\right|_{x=l_0}&=0,
	\end{aligned}
	\end{equation}
which is assumed to hold for each admissible variation $\delta w(x,t)$ of class $C^2\left([0,l]\times [0,\tau]\right)$, $\tau>0$, satisfying the boundary conditions
$$\delta w|_{t=0}=\delta w|_{t=\tau}=\delta w|_{x=0}=\delta w|_{x=l}=0.$$
The function $w(x,t)\in C^2\left([0,l]\times [0,\tau]\right)$ represent the transversal deflection of the beam at a point $x$ and time $t$,
$w(0,t)=w(l,t)=0$; $\rho$, $E$, $I$, $m$, $\varkappa$, and $l$ are positive mechanical parameters, $l_0\in(0,l)$ is the point of the rigid body attachment, piecewise continuous functions $\psi_i(x)$ encode the actuators placement, ${{\rm supp}\,\psi_i\cap\{0,l_0,l\}=\emptyset}$, $i=\overline{1,k}$.

It has been shown in~\cite{KZB2021} that the parameters $\omega_j$, appearing in the operator $A$ in~\eqref{linsys}, are determined in terms of eigenvalues of the spectral problem
\begin{equation}\label{spectral}
	\begin{gathered}
		\frac{d^4}{dx^4}W(x) = \omega^2 \frac{\rho}{EI}W(x),\quad x\in (0,l)\setminus \{l_0\},\\
		W(0)=W(l)=0,\; W''(0)=W''(l)=0,\;W\in C^2[0,l],\\
		W'''(l_0-0)-W'''(l_0+0)=\frac{\varkappa-\omega^2 m}{EI}W(l_0).
	\end{gathered}
\end{equation}
Equations~\eqref{spectral} are derived by separation of variables in the homogeneous part of~\eqref{VF}.

Precisely, the eigenfrequencies ($\omega_j^2$ in the current notation) can be calculated as solutions of the following characteristic equation:
\begin{equation}\label{chareq}
	\Delta(\omega) = 0,
\end{equation}
where
\begin{equation*}
	\begin{aligned}
		\Delta(\omega) &= \frac{m}{4 \mu \rho} \left\{ (\cosh \mu (l-2 l_0) - \cosh \mu l ) \sin \mu l \right.\\
		&\left.+ (\cos \mu(l-2 l_0)-\cos \mu l ) \sinh \mu l \right\}
		- \frac{\sin \mu l \sinh \mu l}{\mu^2}\\
		&+\frac{\varkappa}{4 EI \mu^5}\left\{
		(\cosh \mu l -\cosh\mu(l-2l_0))\sin \mu l \right. \\
		&\left.- (\cos \mu (l-2 l_0)+\cos \mu l) \sinh \mu l \right\}.
	\end{aligned}
\end{equation*}

It has been proved in~\cite{ZK2021} and~\cite{KZB2021} that $\omega_j$ form an increasing sequence of distinct positive real numbers, so Assumption~\ref{assumption1} is fulfilled. Moreover (see~\cite{KZ2021}), the eigenvalues $\omega_j^2$ of~\eqref{spectral} satisfy the following growth condition: let $Q[a,b)$ denote the number of terms of the sequence $\{\omega_j\}_{j\in \mathbb N}$ in the interval $[a,b)$, then
\begin{equation*}
	\underset{y\to\infty}{\lim\sup}\;\underset{\tilde y\to\infty}{\lim\sup}\;\frac{Q[y,y+\tilde y)}{\tilde y}=0.
\end{equation*}

\begin{lemma}\label{prop}
	Let $\omega_j$ be the solutions of equation~\eqref{chareq}, ${j=1,2,\dots}$, then the system of exponents $\{e^{\pm i\omega_jt}\}_{j=1}^\infty$ is linearly independent on $L^2[0,\tau]$ for any $\tau>0$.
\end{lemma}
The assertion of this lemma follows directly from~\cite[Theorem~1.2.17]{K1992}.

The eigenfunctions $W_j(x)$, $j\in{\mathbb N}$ of the corresponding spectral problem form a linearly independent orthogonal system with respect to the inner product $$\left<W_i,W_j\right>_{\tilde H}=\int_0^l \rho W_i(x)W_j(x)\,dx + m W_i(l_0) W_j(l_0).$$

Let us take the eigenfunctions $W_1(x),W_2(x),\dots$, corresponding to eigenvalues $\omega_1^2,\omega_2^2,\dots$, and consider the linear mani\-fold ${\cal S}={\rm span}\{W_1(x),W_2(x),\dots\}$. Take $w$ and $\delta w$ from ${\cal S}$, namely, $w(x,t)=\sum\limits_{i=1}^\infty q_iW_i(x)$ and $\delta w(x,t)=W_j(x)$, and substitute them into~\eqref{VF}. Here $q_i$ are coefficients of the linear combination ($q_i$ depend on $t$). After performing the integration by parts and taking into account the orthogonality of the eigenfunctions, we obtain the following infinite system of ordinary differential equations:
\begin{equation}\label{prj}
	\ddot q_j + \omega_j^2 q_j = \frac{W_j(l_0)}{\|W_j\|^2_{\tilde H}}u_0+\sum_{i=1}^k \frac{\int_0^l\psi_i(x) W_j''(x)\,dx}{\|W_j\|^2_{\tilde H}}u_i,\; j=1,2,\dots \;.
\end{equation}
The system~\eqref{prj} in this interpretation represents the orthogonal projection of the beam-body equations of motion onto the infinite-dimensional linear manifold ${\cal S}$.
In this outline, we treat data $c_{sj}$ as parameters of the output signals provided by sensors located at points $x=l_s$ of the beam,
$$c_{1j} = W_j(l_0),\quad\text{and}\quad c_{sj}=W_j''(l_{s-1}),\quad s=\overline{2,r},$$
the control $u_0$ is treated as the  force applied to the rigid body at $x=l_0$, and $u_1,\dots,u_k$ are treated as actions supplied by $k$  piezoelectric actuators,
$$b_{j0}=\frac{W_j(l_0)}{\|W_j\|^2_{\tilde H}},\;\; b_{jp}=\frac{\int_0^l\psi_p(x) W_j''(x)\,dx}{\|W_j\|^2_{\tilde H}},\; p=1,2,\dots,k.$$

Then we denote ${\xi_j=\omega_jq_j}$ and ${\eta_j=\dot q_j}$ respectively, $j\in\mathbb{N}$, and assume that ${\xi=(\xi_1,\xi_2,\dots)^T\in \ltwo}$, ${\eta=(\eta_1,\eta_2,\dots)^T\in \ltwo}$, which leads  to the system of first-order ordinary differential equations written in the abstract form as~\eqref{linsys}--\eqref{output}.

Now the assertion of Theorem~\ref{th_observer_linsys} can be adjusted for the considered flexible system.

\begin{theorem}
	Assume that the components $\omega_1,\omega_2,\dots$ of the operator $A$ in~\eqref{linsys} are the solutions of equation~\eqref{chareq},
and the operator $F$ in~\eqref{error} is defined by~\eqref{obsmatr} with arbitrary positive gain parameters $\gamma_s$, $s=\overline{1,r}$.
 If Assumption~\ref{assumption4} is satisfied, then the trivial solution $e=0$ of system~\eqref{error} is asymptotically stable.
\end{theorem}

\begin{proof}
	Instead of step-by-step proof, we will just underline the key difference of this case from Theorem~\ref{th_observer_linsys}.
	Note that system~\eqref{error} degrades into
	\begin{equation}\label{err}
		\dot e(t)=Ae(t)
	\end{equation}
	on ${\rm Ker}\,C$.
	
	Suppose initial the conditions
	\begin{equation}\label{IC}
		e(0)=e_0\in H.
	\end{equation}
	As a consequence of Lemma~\ref{le_generator}, the Cauchy problem~\eqref{err},~\eqref{IC} is well-posed on $H$.
	The general solution of~\eqref{err},~\eqref{IC} can be written as
	\begin{equation}\label{err_sol}
		e(t)=\exp(tA)e_0.
	\end{equation}
	The substitution of~\eqref{err_sol} into~\eqref{ker} leads to
	$		C\exp(tA)e_0\equiv0$ or, equivalently,
	\begin{equation}\label{err_sol1}
		\sum\limits_{j=1}^\infty c_{sj}\left(\xi_j(0)\cos\omega_jt+\eta_j(0)\sin\omega_jt\right)=0,\quad s=\overline{1,r}.
	\end{equation}
	
	As mentioned in Lemma~\ref{prop}, the system of functions $\{\cos\omega_jt,\sin\omega_jt\}_{j=1}^\infty$ is linearly independent on  $L^2[0,\tau]$ for any $\tau>0$, since $\omega_j$ satisfy~\eqref{chareq}. Taking into account condition~\eqref{outputmatr}, the property~\eqref{err_sol1} holds only for ${\xi_j(0)=0}$, ${\eta_j(0)=0}$, ${\forall\; j=1,2,\dots}$, i.e. for $e_0\equiv0$. That is, the only solution of system~\eqref{error} on ${\rm Ker}\,C$ is the trivial one. So, the only invariant subset of ${\rm Ker}\,C$ under the action of the semigroup $\{e^{tA}\}_{t\geq0}$ is the point $e=0$.
	
	Thus, we see that Assumption~\ref{assumption3} is guaranteed by the asymptotic distribution of the eigenvalues. So, the conditions of Theorem~1 are satisfied that leads us to the conclusion about the asymptotic stability of the trivial equilibrium of~\eqref{error}.
\end{proof}

\section{OBSERVATION ERROR CONVERGENCE}\label{sec_simulations}
In order to illustrate the observation error dynamics,
we perform the integration of truncated systems~\eqref{linsys}--\eqref{output},~\eqref{observer}, and~\eqref{error} with the coordinate indices $j=1,2,...,N$, where $N$ is a given natural number.
Let us denote by $e_N(t)$ the solution of the corresponding finite-dimensional system~\eqref{error} with $j=1,\dots,N$,
and its Euclidean norm by
$$\|e_N(t)\|=\sqrt{\sum\limits_{j=1}^N(\Delta_j^2+\delta_j^2)}.
$$

We present the results of numerical simulations for the flexible beam model of length ${l=1.875\,\text{m}}$ with the rigid body attached at the point ${l_0=1.378\,\text{m}}$. This choice of mechanical parameters is described in~\cite{KZB2021}. The outputs are assumed to be provided by four piezoelectric sensors located at ${l_1=0.075\,\text{m}}$, ${l_2=0.716\,\text{m}}$, ${l_3=1.128\,\text{m}}$, and ${l_4=1.555\,\text{m}}$. The numerical integration has been carried out in Maple for the truncated error dynamics with ${N=6}$, ${N=16}$, and ${N=40}$ modes of vibration. The initial conditions are taken   as ${\Delta_j(0)=\delta_j(0)=\frac1{j\,\omega_j}}$.

Figure~1 depicts the first six modal deflections $\Delta_i(t)$ with the gain parameters $\gamma_i=6$, $i=1,\dots,6$.
Figures~2 and~3 depict, respectively, $\|e_{16}(t)\|^2$ and $\|e_{40}(t)\|^2$ with the observer gain parameters $\gamma_i=0.8$, $\gamma_i=6$, and $\gamma_i=12$, ${i=0,\dots,4}$.

\begin{figure}[thpb]
	\centering
		\includegraphics[scale=0.3]{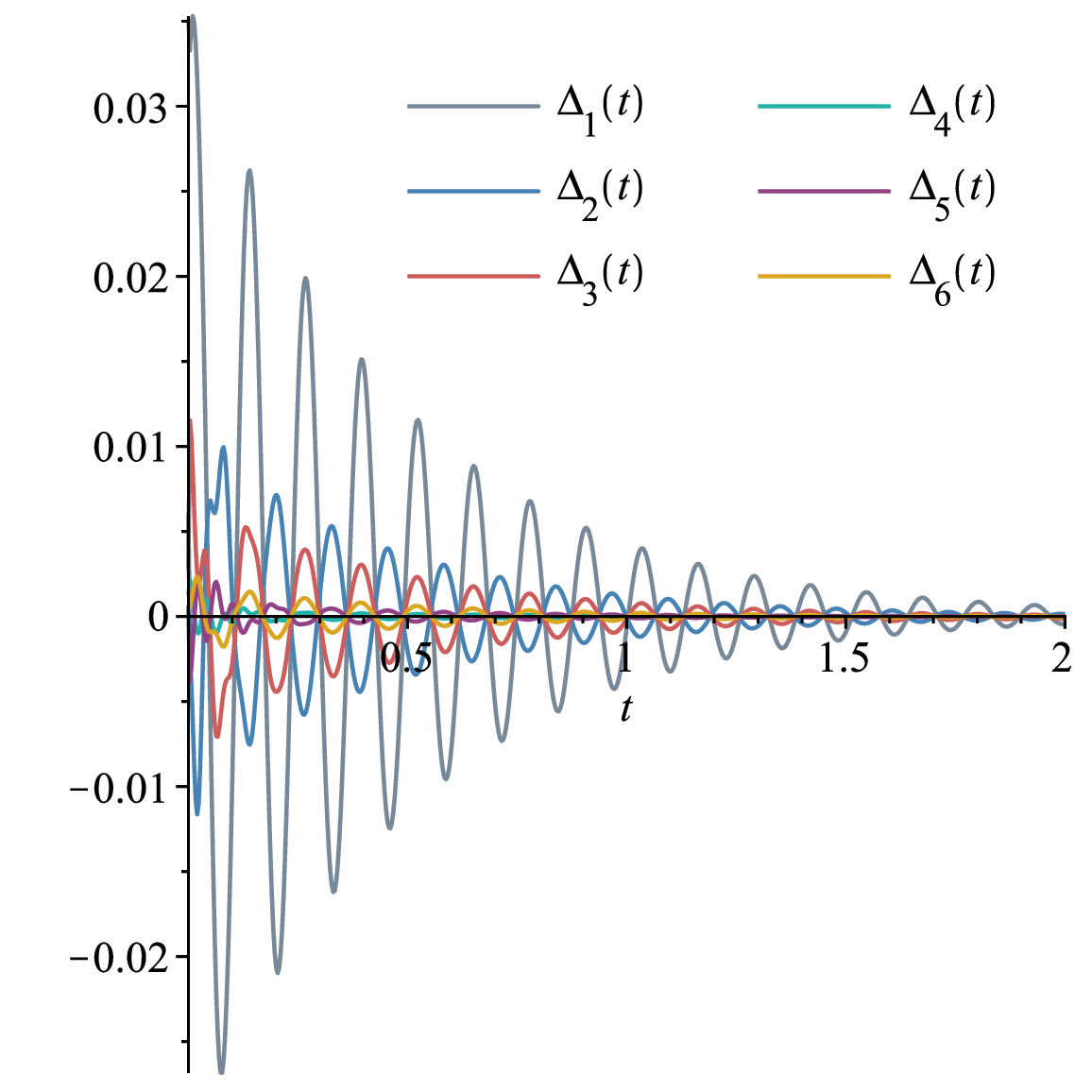}
		\caption{The time plots of $\Delta_1(t),\dots,\Delta_6(t)$}
	\label{f_delta1N06g6}
\end{figure}

\begin{figure}[thpb]
	\centering
		\includegraphics[scale=0.3]{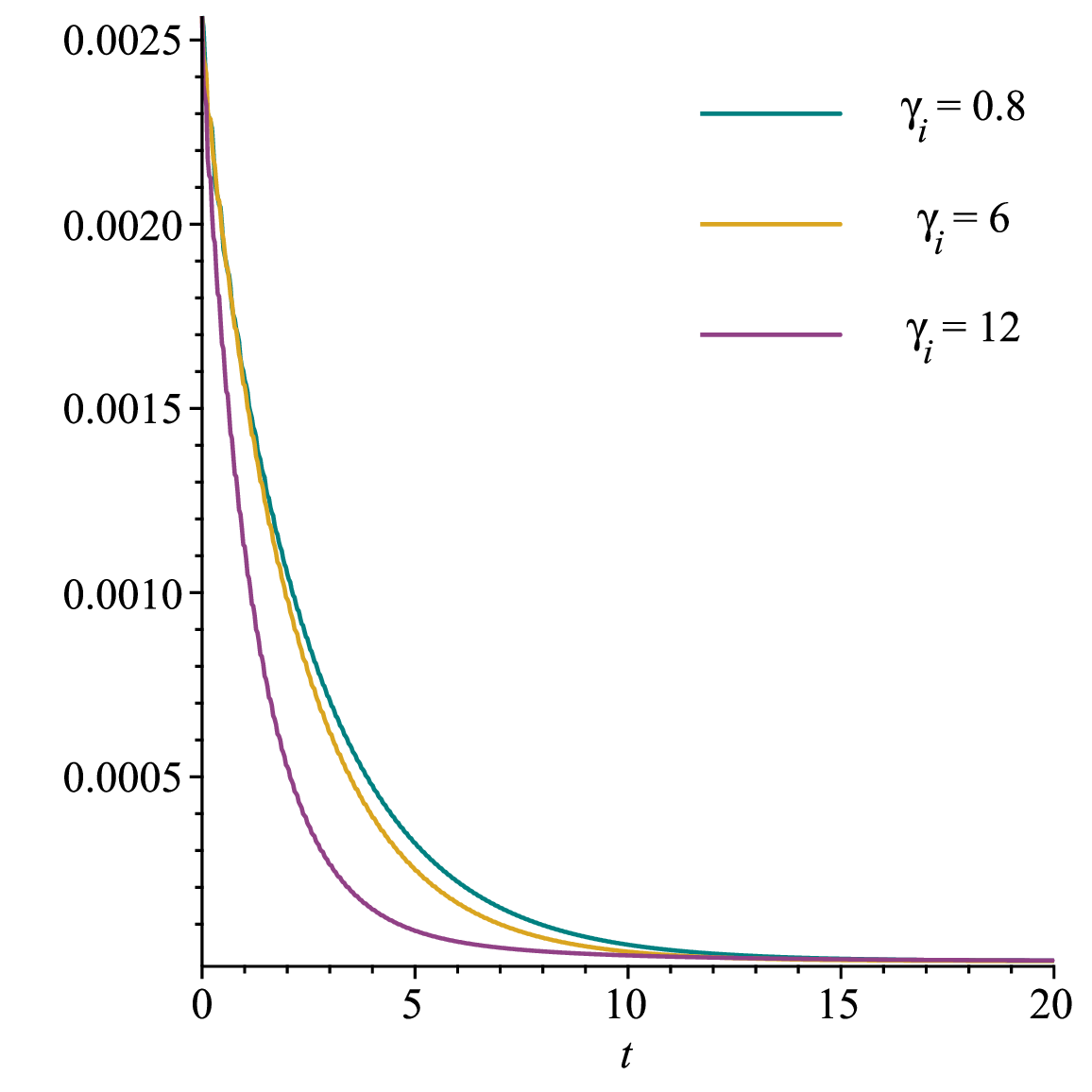}
		\caption{The time plot of $\|e_{16}(t)\|^2$}
	\label{f_enormN16}
\end{figure}

\begin{figure}[thpb]
	\centering
		\includegraphics[scale=0.3]{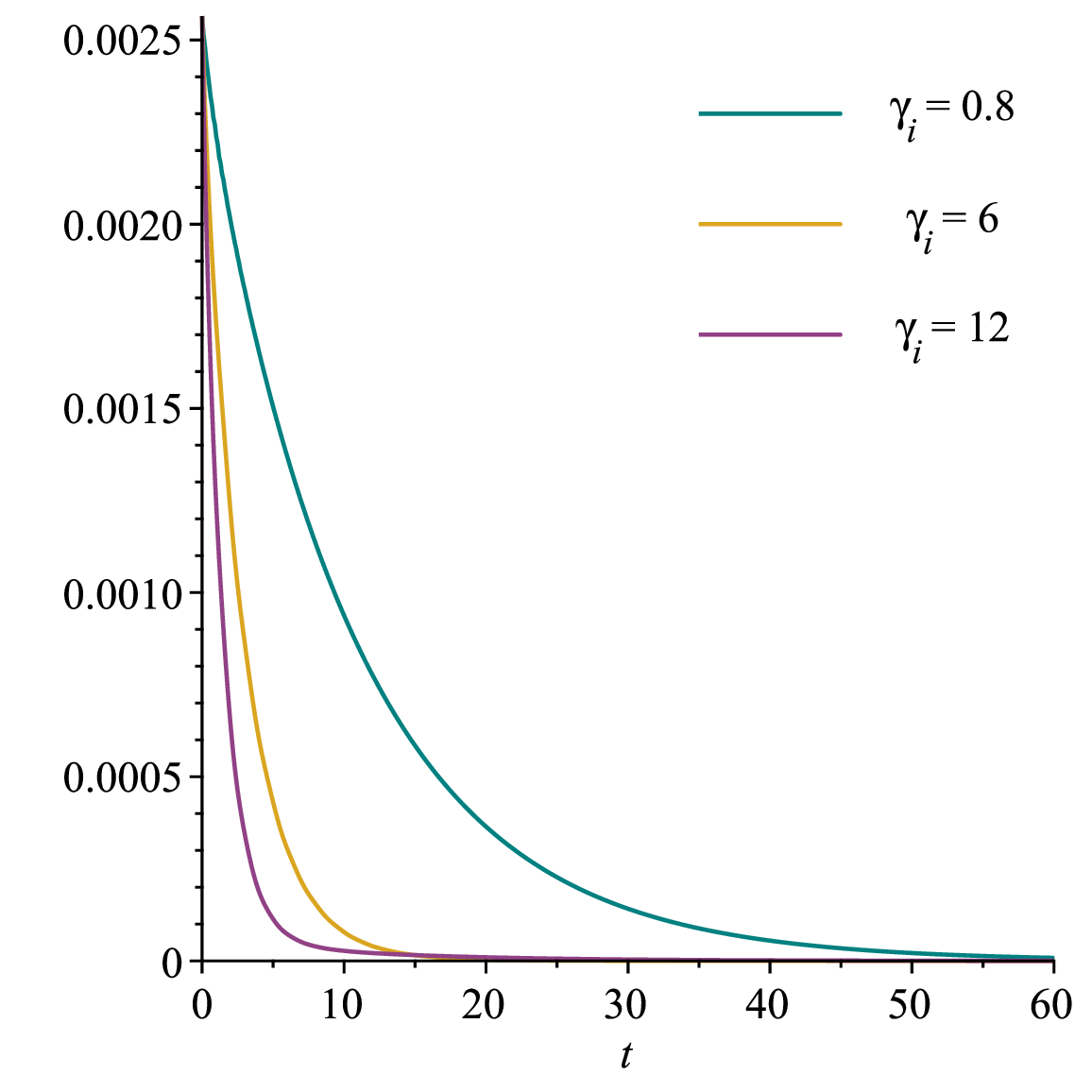}
		\caption{The time plot of $\|e_{40}(t)\|^2$}
	\label{f_enormN40}
\end{figure}

These simulation results illustrate the exponential convergence of the finite-dimensional error dynamics.
We also observe that solutions of the system of larger dimension ($N=40$) decay slower than the solutions with $N=16$ for large values of $t$.

\section{CONCLUSIONS}

We have considered infinite-dimensional Hamiltonian control system~\eqref{linsys}--\eqref{output} and derived the Luenberger-type observer in the form~\eqref{observer}
that allows to asymptotically reconstruct  the full state of the original system by using a finite number of outputs.
The observer design problem has been solved relying on Lyapunov's direct method and the invariance principle,
where the energy-induced Lyapunov functional is taken to measure the observation error.
The proposed observer takes into account the input action and thus admits arbitrarily large inputs. It also permits unbounded outputs as functions of time.
In the considered flexible beam example, the parameters required for the observer design are efficiently  evaluated in terms of eigenvalues and eigenfunctions of the corresponding spectral problem.

In this work, we do not raise the question about the decay rate of the observation error $e(t)$ as $t\to +\infty$.
It appears to be exponential for any finite-dimensional projection of the proposed observation scheme (as discussed in Section~\ref{sec_simulations}),
while in the infinite-dimensional case it may not be exponential (we expect it to be polynomial).
The estimation of the observation error decay rate for the considered class of infinite-dimensional systems is considered as a topic for future investigation.



\end{document}